\documentclass[11pt]{article}
\usepackage{amssymb,amsthm,amsmath,mathrsfs}
\usepackage{prettyref,cite}
\usepackage{hyperref}
\usepackage{graphicx}
\usepackage [a4paper, footskip=1cm, headheight = 16pt, top=2cm, bottom=2.5cm,  right=2cm,  left=2cm ]{geometry}
\newcommand\pref[1]{\prettyref{#1}}
\newrefformat{chp}{\chaptername~\ref{#1}}
\newrefformat{sec}{Section~\ref{#1}}
\newrefformat{thm}{Theorem~\ref{#1}}
\newrefformat{lem}{Lemma~\ref{#1}}
\newrefformat{cor}{Corollary~\ref{#1}}
\newrefformat{con}{Conjecture~\ref{#1}}
\newrefformat{prop}{Proposition~\ref{#1}}
\newrefformat{dfn}{Definition~\ref{#1}}
\newrefformat{exm}{Example~\ref{#1}}
\newrefformat{rem}{Remark~\ref{#1}}
\newrefformat{fig}{\figurename~\ref{#1}}
\newrefformat{tbl}{\tablename~\ref{#1}}
\newrefformat{eq}{$($\ref{#1}$)$}

\newtheorem{pro}{Proposition}
\newtheorem{cor}[pro]{Corollary}

\newtheorem{thm}[pro]{Theorem}
\newtheorem{precon}[pro]{Conjecture}
\newtheorem{predef}[pro]{Definition}
\newtheorem{prerem}[pro]{Remark}
\newtheorem{alphthm}{Theorem}
\renewcommand{\thepro}{\arabic{pro}}

\newenvironment{rem}{\begin{prerem}\rm}{\hfill $\blacktriangle$\end{prerem}}

\newenvironment{exm}
{\par\medskip\refstepcounter{pro}\noindent{\bf Example \thepro\ }}
{\par\hfill $\square$\par\medskip\noindent}

\DeclareMathOperator{\lbc}{lbc}
\DeclareMathOperator{\bc}{bc}
\DeclareMathOperator{\cc}{cc}
\DeclareMathOperator{\kc}{lcc}
\DeclareMathOperator{\lr}{lr}
\newcommand{\lcc}[1]{\kc(#1)}
\newcommand{\wt}[1]{\tilde{\omega}(#1)}
\newcommand{\rep}[1]{\mathfrak{#1}}
\newcommand{\power}[1]{\mathcal{P}(#1)}
\newcommand{\bq}{\begin{equation}}
\newcommand{\eq}{\end{equation}}
\begin{document}
\title{Local Clique Covering of Graphs}
\author{Ramin Javadi, Zeinab Maleki and Behanz Omoomi \\[4mm]
Department of Mathematical Sciences\\
Isfahan University of Technology, 84156-83111, Isfahan, Iran}
\date{}
\maketitle
\begin{abstract}
A \textit{$k-$clique covering} of a simple graph $G$, is an edge covering of $G$ by its cliques such that each vertex is contained in at most $k$ cliques. The smallest $k$ for which $G$ admits a $k-$clique covering is called \textit{local clique cover number} of $G$ and is denoted by $\lcc G$. Local clique cover number can be viewed as the local counterpart of the clique cover number which is equal to the minimum total number of cliques covering all edges. In this paper, several aspects of the problem are studied and its relationships to other well-known problems are discussed. Moreover, the local clique cover number of claw-free graphs and its subclasses are notably investigated. In particular, it is proved that local clique cover number of every claw-free graph is at most $c\Delta/\log \Delta$, where $\Delta$ is the maximum degree of the graph and $c$ is a universal constant. It is also shown that the bound is tight, up to a constant factor. Furthermore, it is established that local clique number of the linear interval graphs is bounded by $\log\Delta+1/2 \log\log\Delta+O(1)$.  Finally, as a by-product, a new Bollob\'as-type inequality is obtained for the intersecting pairs of set systems.
\end{abstract}
\section{Introduction}\label{sec:intro}
Throughout the paper, all graphs are finite and simple (unless it is clearly mentioned) and the term \textit{clique} stands for  both  a set of pairwise adjacent vertices and also the corresponding induced complete subgraph. In addition, by a \textit{biclique} we mean a complete bipartite subgraph. 
In the literature, different variants of edge covering of graphs have been explored. Among them, the clique covering and biclique  covering are widely studied. 
A \textit{clique $($resp. biclique$)$ covering} of the graph $G$ is a family $\mathcal{C}$ of cliques (resp. bicliques) of $G$ such that  every edge of $G$ belongs to at least one clique (resp. biclique) in $\mathcal{C}$.
The  \textit{clique} (\textit{resp. biclique}) \textit{cover number} of $G$, denoted by $\cc(G)$ (resp. $\bc(G)$), is defined as the smallest number of cliques (resp. bicliques) in a clique covering of $G$. These concepts turn out to have several  relations and applications to a large variety of theoretical and  applied  problems including set intersection representations of graphs, communication complexity of boolean functions and encryption key management. For a review of the  clique and biclique covering see~\cite{jukna1, jukna2,pullman}.

In contrast to the clique covering problem which is aimed at minimizing the ``total'' number of cliques comprising a clique covering, in this paper, we are interested in minimizing the maximum number of cliques which are incident with ``each vertex''. Let us make the notion more accurate. 
Given a clique covering $\mathcal{C}$  of a graph $G=(V,E)$, for every vertex $x\in V$, the \textit{valency} of $x$ (with respect to $\mathcal{C}$), denoted by $v_\mathcal{C}(x)$, is defined to be the number of cliques inside $\mathcal{C}$ containing $x$. The \textit{valency} of the clique covering $\mathcal{C}$ is the maximum valency of all vertices of $G$ with respect to $\mathcal{C}$. The clique covering $\mathcal{C}$ is called a \textit{$k-$clique covering} if its valency is at most $k$, i.e. every vertex of $G$ belongs to at most $k$  cliques within  $\mathcal{C}$. Among all clique coverings of $G$, we are interested in finding a clique covering of the minimum valency. 
The smallest number $k$ being the valency of a clique covering of $G$, is called the \textit {local clique cover number} of $G$ and is denoted by $\lcc{G}$. In fact,
\[\lcc G:=\min_{\mathcal{C}} \max_{x\in V} v_{\mathcal{C}}(x),\]
where the minimum is taken over all clique coverings of $G$.

In other words, $\lcc {G}$ is the minimum number $k$ for which $G$ admits a $k-$clique covering. The concept of $k-$clique covering has been introduced in \cite{skums}, during the study of the edge intersection graphs of linear hypergraphs (see Section~\ref{sec:rel}). The problem of finding the local clique cover number of a graph appears to have a number of interesting interconnections and interpretations to some other well-known problems. We will discuss these relationships in Section~\ref{sec:rel}.

In this paper, the main effort is devoted to investigating  the local clique cover number of ``claw-free'' graphs. A \textit{claw-free graph} is a graph having no complete bipartite graph $K_{1,3}$ as an induced subgraph. Also, a \textit{quasi-line graph} is a graph where the neighbours of each vertex are union of two cliques. Quasi-line graphs may be considered as the generalization of line graphs as well as claw-free graphs as the generalization of quasi-line graphs. This has been a natural and recently well-studied question that which properties of line graphs can be extended to quasi-line graphs and then to all claw-free graphs (see e.g.~\cite{seymour,seymourcol}). Our motivation for picking the class of claw-free graphs to study is twofold. Firstly, we will see in Corollary~\ref{cor:line} that the local clique cover number of line graphs is at most~$2$. This arises the natural question that how large the local clique cover number of a quasi-line graph and a claw-free graph can be. We will answer this question in Section~\ref{sec:clawfree}.

Secondly, the local clique cover number being increasing in terms of the induced subgraph partial ordering,  
motivates us to define the following parameter related to $G$,
\[\alpha_l(G):=\max \{ t \ :\ K_{1,t}\mbox{ is an induced subgraph of }G \}.\]
An \textit{independent set} is a subset of mutually non-adjacent vertices. In fact,  $\alpha_l(G)$ is the size of the maximum independent set within the neighbourhood of a vertex. Then, in some sense, the parameter $\alpha_l(G)$ can be thought of as the \textit{maximum local independence number} of $G$. If $K_{1,t}$ is an induced subgraph of $G$, then $t=\lcc{K_{1,t}}\leq \lcc{G}$ and thereby, $\alpha_l(G)\leq \lcc{G}$. Though this bound can be tight, e.g. for the cases $\lcc G=\Delta(G)$ (see Proposition~\ref{pro:delta}), it can also be very loose. For instance, let $G_t=K_{2,2,\ldots,2}$ be the $t-$partite complete graph. Then $\alpha_l(G_t)=2$, however, we will see that $\lcc{G_t}> (1/2) \log t$ (see the note following Proposition~\ref{pro:kneser}). This arises the question that for every fixed $t$, how large the lcc of a graph $G$ can be, whenever $\alpha_l(G)\leq t$. In fact, investigating the lcc of claw-free graphs (as the graphs $G$ with $\alpha_l(G)\leq 2$) can be perceived as the first step towards answering the question for general $t$.

It should be noted that, on the same line of thought, Dong et al. \cite{dong} have proposed the local counterpart of the biclique cover number. The \textit{local biclique cover number} of a graph $G$, denoted by $\lbc(G)$, is defined as the smallest $k$ for which $G$ admit a $k-$biclique covering, i.e. a biclique covering where each vertex is incident with at most $k$ of the bicliques comprising the covering.

In Section~\ref{sec:lig}, the local biclique cover number will be applied to compute the lcc of the linear interval graphs. For this reason, let us recall an algebraic interpretation of biclique coverings and introduce an analogous interpretation for the local variant.
Orlin \cite{orlin} has presented an interpretation of biclique covering of bipartite graphs, using boolean rank. The \textit{rank} of a matrix $A$, denoted by $r(A)$, is defined as the smallest $k$, for which there exist vectors $x_i, y_i$, $1\leq i\leq k$, satisfying
\begin{equation} \label{eq:rank}
A=\sum_{i=1}^k x_i y_i^T.
\end{equation}
Let $A$ be a binary matrix. If we confine ourself to binary vectors $x_i,y_i$ and use the boolean arithmetic (i.e. everything as usual except $1+1=1$), then the minimum $k$ for which \eqref{eq:rank} holds, is called the \textit{boolean rank} of $A$ and is denoted by $r_B(A)$. It is easy to see that the boolean rank of $A$, $r_B(A)$, is the smallest $k$ for which one can cover all one entries of $A$ by $k$ all-ones submatrices \cite{watts}. We call an all-ones submatrix of $A$, as a \textit{rectangle} of $A$. Now, assume that $A$ is the bipartite adjacency matrix of a bipartite graph $G=(X,Y)$ (i.e. $A$ is an $|X|\times |Y|$ matrix where $A_{xy}$ is equal to $1$ if and only if $x\in X$ is adjacent to $y\in Y$). In this case, a rectangle in $A$ corresponds to a biclique in $G$. As a result, $r_B(A)$ turns out to be equal to the minimum number of bicliques covering all edges of $G$, i.e. $r_B(A)=\bc(G)$ \cite{watts,orlin}.

In the same vein, one might give an analogous interpretation for the local biclique cover number. For a binary matrix $A$, we define \textit{local boolean rank} of $A$, denoted by $\lr_B(A)$, as the smallest $k$ such that one can cover all  one entries of $A$ by all-ones submatrices (rectangles), in a way that each row and column of $A$ appears  in at most $k$ rectangles. Now if $A$ is the bipartite adjacency matrix of a bipartite graph $G$, then the preceding argument implies that $\lr_B(A)=\lbc(G)$. We will use the local boolean rank in Section~\ref{sec:lig}.

Now, let us give an overview of the organization of forthcoming sections. In Section~\ref{sec:rel}, we review the relationships between the local clique cover number and three well-known problems, namely, line hypergraphs, intersection representation and Kneser representation. In Section~\ref{sec:basic}, we derive some basic bounds for the lcc in terms of the maximum degree and the maximum clique number. In addition, we characterize all the graphs $G$ where $\lcc G=\Delta(G)$.
Section~\ref{sec:clawfree} is set up to compute lcc of the claw-free graphs. In particular, we prove that lcc of a claw-free graph is at most $c\Delta/\log\Delta$, where $\Delta$ is maximum degree of the graph and $c$ is a constant. Moreover, we prove that the bound is the best possible upper bound up to a constant factor. In Section~\ref{sec:lig}, a special class of claw-free graphs, namely linear interval graphs, is considered and particularly, it is proved that lcc of a linear interval graph is at most $\log\Delta+1/2 \log\log\Delta+O(1)$. Finally, in Section~\ref{sec:bol}, a result of Section~\ref{sec:lig} is deployed to prove a new bollob\'as-type inequality for a pair of set systems.
\section{Interactions and interpretations}\label{sec:rel}
As we mentioned before, the local clique cover number may be interpreted as a variety of different invariants of the graph and the problem relates to a number of other well-known problems. In the following, we provide an overview of three important problems associated to the local clique cover number, namely, line hypergraphs, intersection representation and Kneser representation of graphs.
\vspace{-2mm}
\paragraph{Line hypergraphs.}
Given a hypergraph $H=(V,\mathcal{F})$, the \textit{line graph} or \textit{edge intersection graph} of $H$, denoted by $L(H)$, is a simple graph whose vertices correspond to the edges of $H$ and  a pair of vertices in $L(H)$ are adjacent if and only if their corresponding edges in $H$ intersect.
For an arbitrary graph $G$, the inverse image $L^{-1}(G)$ is the set of all hypergraphs $H$ where $L(H)=G$.  In \cite{berg89}, the concept has been described in terms of clique covering. For this, let $\mathcal{C}$ be a clique covering for the graph $G$ and for each vertex $x\in V(G)$, let $\mathcal{C}_x$ be the set of all cliques in $\mathcal{C}$ which contain $x$. It has been proved that an arbitrary graph $G$ is a line graph of a hypergraph $H$ if and only if there is a clique covering $\mathcal{C}$ for $G$ such that $H=(\mathcal{C}, \{\mathcal{C}_x\ :\  x\in V(G)\})$ \cite{berg89}. 
From this, one may deduce that every simple graph is the line graph of a hypergraph. Also, the problem of recognizing a class of line hypergraphs is reduced to investigating the clique covering of graphs.

A hypergraph $H$ is called  \textit{$k-$uniform} if all its edges have the same cardinality $k$. The class of line graphs of $k-$uniform hypergraphs is denoted by $L_k$. Note that, in every $k-$clique covering, one may make the valency of all vertices equal to $k$ by adding some dummy single-vertex cliques. Thus, one may see that a graph $G$ belongs to the class $L_k$ if and only if it admits a $k-$clique covering (see \cite{skums}, for more details). Hence, the following equality holds,
\begin{equation}\label{eq:hyper}
\lcc{G}=\min\{k \ :\  G\in L_k\}.
\end{equation}
Consequently, the problem ``Is $\lcc G\leq k$?'' reduces to ``Is $G\in L_k$?'', i.e. ``Does there exist a $k-$uniform hypergraph whose line graph is isomorphic to $G$?''. It is clear that $\lcc{G}=1$ if and only if $G$ is a disjoint union of cliques.  Also, by \eqref{eq:hyper}, we have the following corollary.
\begin{cor}\label{cor:line}
For every graph $G$, $\lcc{G}\leq 2$ if and only if $G$ is the line graph of a multigraph.
\end{cor}
The class $L_2$ turns out to have a characterization by a list of $7$ forbidden induced subgraphs and a polynomial time algorithm has been found for the recognition of $G\in L_2$ (see \cite{bermond,omoomi,lehot}). In contrast to the case $k=2$, the situation of the case $k\geq 3$ is completely different. Lov{\'a}sz in~\cite{lovasz} has proved that there is no characterization of the class $L_3$ by a finite list of forbidden induced subgraphs. Also, it has been proved that the decision problems $G \in   L_k$ for fixed $k \geq 4$ and the problem of recognizing line graphs of $3-$uniform hypergraphs without multiple edges are $NP$-complete  \cite{poljak}. This leads us to the following hardness results for lcc.
\begin{cor}\label{l3}
\begin{itemize}
\item[\rm (i)] The decision problem $\lcc{G} \leq 2$ is polynomially solvable.
\item[\rm (ii)] For every fixed $k\geq 4$, the decision problem $\lcc{G}\leq k$ is an $NP$-complete problem.
\item[\rm (iii)] The decision problem that if there exists a $3-$clique covering for $G$, where no two distinct vertices appear in exactly the same set of cliques, is an $NP$-complete problem.
\end{itemize} 
\end{cor}
\noindent Also note that $NP$-completeness of the decision problem $\lcc{G} \leq 3$, in general case, remains open.
\paragraph{Intersection representation.}
An \textit{intersection representation} for graph $G=(V,E)$ is a representation of each vertex by a set, such that every two distinct vertices are adjacent if and only if their corresponding sets intersect. In other words, it is a function $\rep{R}:V\to \power{L}$, where $L$ is a set of labels, such that for every two distinct vertices $x,y\in V$,
\[x\sim y\ \mbox{ if and only if }\ \rep{R}(x)\cap \rep{R}(y)\neq \emptyset.\]
For each $i\in L$, the vertices being represented by the sets containing $i$ form a clique in $G$. On the other hand, every clique covering $\mathcal{C}$ induces an intersection representation which assigns to each vertex $x$ the set $\mathcal{C}_x$ (see above).
This sets up a one-to-one correspondence between the clique coverings of $G$ and the intersection representations for $G$ (see e.g. \cite{mckee}, for more details). A \textit{$k-$representation} is an intersection representation $\rep{R}$ such that for each $x\in V$, $\rep{R}(x)$ is of size at most $k$. As a matter of fact, one may conclude that $\lcc G$ is the minimum $k$ for which $G$ admits a $k-$representation. Indeed,
\begin{equation}
\lcc{G}=\min_{\rep{R}} \max_{x\in V} |\rep{R}(x)|.
\end{equation}
\paragraph{Kneser representation.}
Given positive integers $n$ and $k$, $n\geq 2k$, the Kneser graph with paremeters $n,k$, denoted by $KG(n,k)$, is the graph with the vertex set $[n]^k$, the set of all $k-$subsets of $[n]:=\{1,\ldots,n\}$, such that a pair of vertices are adjacent if and only if the corresponding subsets are disjoint. It can be seen that every graph is an induced subgraph of a Kneser graph \cite{kneser}. Hamburger et al. have proposed the question that what is the smallest $k$ for which $G$ is the induced subgraph of a Kneser graph $KG(n,k)$, for some integer $n$. The minimum $k$ for which there exists some integer $n$ such that $G$ is the induced subgraph of $KG(n,k)$  is called the \textit{Kneser index} of $G$ and is denoted by $\iota^K(G)$ \cite{kneser}.

Let us assume that $G$ admits a $k-$representation $\rep{R}$ which is \textit{injective}. Then, one can add dummy new labels to the sets $\rep{R}(x)$, $x\in V$, in order to make all of them of the same size $k$. This shows that each vertex can be represented by a distinct $k-$subset of $L$, where the sets corresponding to adjacent vertices intersect. Hence, the compliment of $G$, $\overline{G}$, is an induced subgraph of $KG(n,k)$. This leads us to the fact that an injective $k-$representation for $G$ exists if and only if  $\iota^K(\overline{G}) \leq k$. Not all $k-$representations of $G$ are injective, however, if $G$ admits a $k-$representation, then one may find an injective $(k+1)-$representation for $G$ by adding to each set, a single new label.

A pair of adjacent vertices $x,y\in V$, are called \textit{twins}, if $N(x)\setminus \{y\}=N(y)\setminus \{x\}$, where $N(x)$ stands for the set of neighbours of $x$. A graph $G$ is called \textit{twin-free} if it has no twins. Every intersection representation of a twin-free graph is indeed injective, because distinct vertices should be represented by distinct sets of labels. 

The above arguments imply the following proposition relating the lcc to the Kneser index.
\begin{pro}\label{pro:kneser}
For every graph $G$, we have $\lcc{G} \leq \iota^K(\overline{G}) \leq \lcc{G}+1$ and provided $G$ is twin-free, $\lcc{G}=\iota^K(\overline{G})$.
\end{pro}
For instance, it is shown in \cite{kneser} that if the graph $G$ contains a subgraph induced by a matching of size $t$, and $t>\binom{2k}{k}$, for some $k$, then $\iota^K(G)> k$. This implies that $\iota^K(G)>(1/2) \log t$. As a consequent, by Proposition~\ref{pro:kneser},  if $G_t=K_{2,2,\ldots,2}$ is the $t-$partite complete graph, we have $\lcc{G_t}>(1/2) \log t$.

\begin{rem}\label{rem:twinfree}
Every graph $G$ has a twin-free induced subgraph $H$ for which $\lcc G=\lcc H$. To see this, note that being twins endows an equivalence relation on $V(G)$. Let $H$ be the induced subgraph of $G$  obtained by deleting all but one vertex from each of the equivalence classes. It is evident that $\lcc H\leq \lcc G$. On the other hand, every $k-$clique covering for $H$ can be extended to a $k-$clique covering for $G$ substituting every vertex by its corresponding equivalence class. Hence, $\lcc G=\lcc H$, as desired.
\end{rem}
\section{Basic bounds}\label{sec:basic}
In this section, we provide simple lower and upper bounds for $\lcc{G}$ in terms of the maximum degree and the maximum clique number. In addition, we characterize the case when the upper bound meets.
\begin{pro} \label{pro:bounds}
For every graph $G$ with maximum degree $\Delta$ and maximum clique number $\omega$, we have
\bq\label{eq:bound}
\frac{\Delta}{\omega-1} \leq \lcc{G} \leq \Delta.
\eq
\end{pro}
\begin{proof}{
The upper bound is a straightforward result of the fact that all edges of $G$ comprise a $\Delta$-clique covering. For the lower bound, let $k=\lcc G$ and $\mathcal{C}$ be a $k-$clique covering. Fix a vertex $x$ and define
$\mathcal{C}_x:=\{C\in \mathcal{C}\ :\  x\in C\}$. By the definition of lcc, we have $|\mathcal{C}_x|\leq k$. Furthermore, each edge incident with $x$ is contained in some clique $C\in \mathcal{C}_x$. Therefore
\[\deg(x)\leq \sum_{C\in \mathcal{C}_x} |C\setminus \{x\}|\leq |\mathcal{C}_x|(\omega-1)\leq k (\omega-1)=\lcc G (\omega-1).\]
Being the vertex $x$ arbitrary, the desired bound follows.
}\end{proof}
Using the above proposition, we may determine the exact value of $\lcc{G}$ for triangle-free graphs.
\begin{cor}
For every triangle free graph $G$ with maximum degree $\Delta$, we have $\lcc{G}=\Delta$.
\end{cor}
The following proposition characterize all the graphs for which $\lcc{G}=\Delta$.
\begin{pro} \label{pro:delta}
For the graph $G$ with maximum degree $\Delta$, we have $\lcc{G}=\Delta$ if and only if there exists a vertex $x\in V(G)$ of degree $\Delta$, such that $N(x)$ is an independent set, that is $\alpha_l(G)=\Delta$.
\end{pro}
\begin{proof}

Suppose that there exists a vertex $x\in V(G)$ of degree $\Delta$, such that $N(x)$ is an independent set. Therefore, in every clique covering of $G$,  we need $\Delta$ cliques to cover the edges incident with $x$. Thus, $\lcc G\geq \Delta$.

Conversely, assume that $\lcc{G}=\Delta$. Let $\mathcal{C}$ be a $\Delta-$covering and, subject to this, $\sum_{C\in\mathcal{C}}|C|$ is minimal. Also, let $x$ be a vertex which is contained in $\Delta$ cliques of $\mathcal{C}$. As a matter of fact, $\deg(x)= \Delta$. Otherwise, the edges incident with $x$ can be covered by at most $\Delta-1$ cliques in $\mathcal{C}$. By excluding $x$ from the extra cliques, we obtain a new $\Delta-$covering, contradicting the minimality assumption. The same argument shows that each clique in $\mathcal{C}$ covers at most one edge incident with $x$. Now, it is enough to prove that $N(x)$ is an independent set. Assume, to the contrary, that $y,z\in N(x)$ are adjacent. In this case, one may replace the cliques $\{x,y\},\{x,z\}\in \mathcal{C}$ by the clique $\{x,y,z\}$ to obtain a new $\Delta-$covering, contradicting the minimality assumption. Hence, the assertion holds.
\end{proof}
%

\section{Claw-free graphs} \label{sec:clawfree}
In this section, we focus on the class of claw-free graphs and particularly we concentrate on the question that how large the lcc of a claw-free graph can be (see Section~\ref{sec:intro} for the origins and motivations). In the light of Propositions~\ref{pro:bounds} and \ref{pro:delta}, one can see that the best upper bound for the lcc of a general graph is $\Delta$. Nevertheless, we show that this bound can be asymptotically improved for the claw-free graphs. In this regard, for every integer $k$,  let us define
\[f(k):=\max\{\lcc{G} \ :\  G \mbox{ is claw-free and } \Delta(G)\leq k\}.\]
Now the question is that how  the function $f(k)$ behaves  in terms of $k$. The same question can be also asked for the quasi-line graphs. Let us define
\[g(k):=\max\{\lcc{G} \ :\  G \mbox{ is a quasi-line graph and } \Delta(G)\leq k\}.\]
In the following, we  determine asymptotic behaviors of the functions $f(k)$ and $g(k)$, by proving that for some constants $c_1,c_2$,
\begin{equation}\label{eq:claw}
c_1\ \frac{k}{\log k}\leq g(k)\leq f(k)\leq c_2\ \frac{k}{\log k}.
\end{equation}
The whole of this section is devoted to establish \eqref{eq:claw} (here, we make no attempt to find the best possible constant factors). In order to prove the lower bound, for every integer $k$, one ought to provide a quasi-line graph $G$ where $\Delta(G)\leq k$ and $\lcc{G}\geq c_1\ k/\log k$. This is exactly what we are going to do in the following theorem.
A graph is called \textit{cobipartite} if its complement is bipartite. It is evident that every cobipartite graph is a quasi-line graph.
\begin{thm}
For every integer $n$, there exists a cobipartite graph $G$ on $n$ vertices 
 such that
$$\lcc{G}> \frac{1}{4}(1-o(1)) \frac{n}{\log n},$$
where $o(1)$ tends to $0$, as $n$ goes to infinity.
\end{thm}
\begin{proof}
For fixed  integers $n,t$, $t\leq {{n^2}/{8}}$,   let $G$ be a graph on $n$ vertices whose compliment is a bipartite $n/2\times n/2$ graph (i.e. $V(G)$ is the disjoint union of two cliques of size $n/2$). Also, let $\rep{R}:V(G)\to \power{L}$ be a $t$-representation for $G$ with the label set $L$ (see Section~\ref{sec:rel} for the definition). Then, without loss of generality, we can assume that $|L|\leq n^2/4$, because for each edge between two parts, we need at most one new label in $L$.

Now, on the one hand, the number of all $t$-representations with $n^2/4$ labels for a graph on $n$ vertices, is at most
\[\left[\sum_{i=0}^{t} \binom{n^2/4}{i} \right]^n\leq t^n \left( \frac{en^2}{4t}\right)^{nt},\]
and, on the other hand, the number of all bipartite $n/2\times n/2$ graphs  (with labelled vertices) is $2^{n^2/4}$. We set $t$ such that for sufficiently large $n$,
\begin{equation}\label{eq:ted}
2^{n^2/4} > t^n \left( \frac{en^2}{4t}\right)^{nt}.
\end{equation}
This ensures the existence of a cobipartite graph $G$ which admits no $t$-representation and consequently $\lcc{G}>t$. It only remains to do some tedious computations to check Inequality \eqref{eq:ted}, for $t:=1/4(1-o(1))\ n/\log n$.
\end{proof}
Proving the upper bound in \eqref{eq:claw} is more difficult.  For this purpose,  our approach is to establish an upper bound for the lcc of claw-free graphs, in terms of the maximum degree. In particular, we  prove that for every claw-free graph $G$ with maximum degree $\Delta$, $\lcc{G}\leq c\ \Delta/\log \Delta$, where $c$ is a universal constant. Towards achieving this objective, we deploy a result of  Erd{\H{o}}s et al. about the decomposition of a graph into complete bipartite graphs \cite{erdos}, along with a well-known result of Ajtai et al. in the independence number of triangle-free graphs. Let us recall them in the following two theorems.
\begin{alphthm} {\rm \cite{erdos}} \label{thm:erdos}
The edge set of every graph on $n$ vertices can be partitioned into complete bipartite subgraphs {\rm (}bicliques{\rm )} such that each vertex is contained in at most $c_0\ n/\log n$ of the bicliques, i.e. $\lbc(G)\leq c_0\ n/\log n$, where $c_0$ is a universal constant.
\end{alphthm}
\begin{alphthm} {\rm \cite{ajtai80,ajtai81,shearer}} \label{thm:ajtai}
Every triangle-free graph on $n$ vertices contains an independent set of size $\sqrt{2}/2(1-o(1))\ \sqrt{n\log n}$.
\end{alphthm}
Since the property of being triangle-free is hereditary with respect to subgraph, one can iteratively apply Theorem \ref{thm:ajtai} and omit the large independent sets from the vertex set, thereby obtaining a proper coloring for the graph (see e.g. \cite{kim,nilli}). Therefore, the following theorem is a consequent of Theorem~\ref{thm:ajtai}.
\begin{alphthm} \label{thm:color}
The vertex set of every triangle-free graph on $n$ vertices can be properly colored by at most $2\sqrt{2}(1+o(1))\sqrt{n/\log n}$ colors, such that each color class is of size at most $\sqrt{2}/2 \sqrt{n\log n}$.
\end{alphthm}
 Now, with all these results in hand, we are ready to prove the main theorem of this section.
\begin{thm}\label{thm:claw}
For every claw-free graph $G$ with maximum degree $\Delta$, we have $\lcc{G}\leq c\  {\Delta}/{\log \Delta}$, where $c$ is a universal constant.
\end{thm}
\begin{proof}
Fix a claw-free graph $G=(V,E)$ with maximum degree $\Delta$. We provide a clique covering for $G$, where each vertex is contained in at most $c\ {\Delta}/{\log \Delta}$ cliques.

Let $I\subset V$ be a maximal independent set of vertices in $G$ and fix a vertex $u\in I$. Since $G$ is claw-free, the subgraph of $G$ induced by the neighbourhood of $u$, $G[N(u)]$, has no independent set of size $3$. Thus, its compliment, $\overline{G[N(u)]}$, is a triangle-free graph on at most $\Delta$ vertices. Then, by Theorem~\ref{thm:color}, the vertex set of $\overline{G[N(u)]}$ can be properly  colored by at most $2\sqrt{2}(1+o(1))\sqrt{\Delta/\log \Delta}$ colors, such that each color class is of size at most $\sqrt{2}/2 \sqrt{\Delta\log \Delta}$. Let $C_1, \ldots, C_\ell$ be all color classes in the coloring of $\overline{G[N(u)]}$, thereby obtained (see Figure~\ref{fig:clawfree}). For each $i$,  $1\leq i\leq \ell$, the set $C_i\cup \{u\}$ forms a clique in $G$. Therefore, all edges incident with $u$ can be covered by at most $2\sqrt{2}(1+o(1))\sqrt{\Delta/\log \Delta}$ cliques.

\begin{figure}
\centering
\includegraphics[scale=.65]{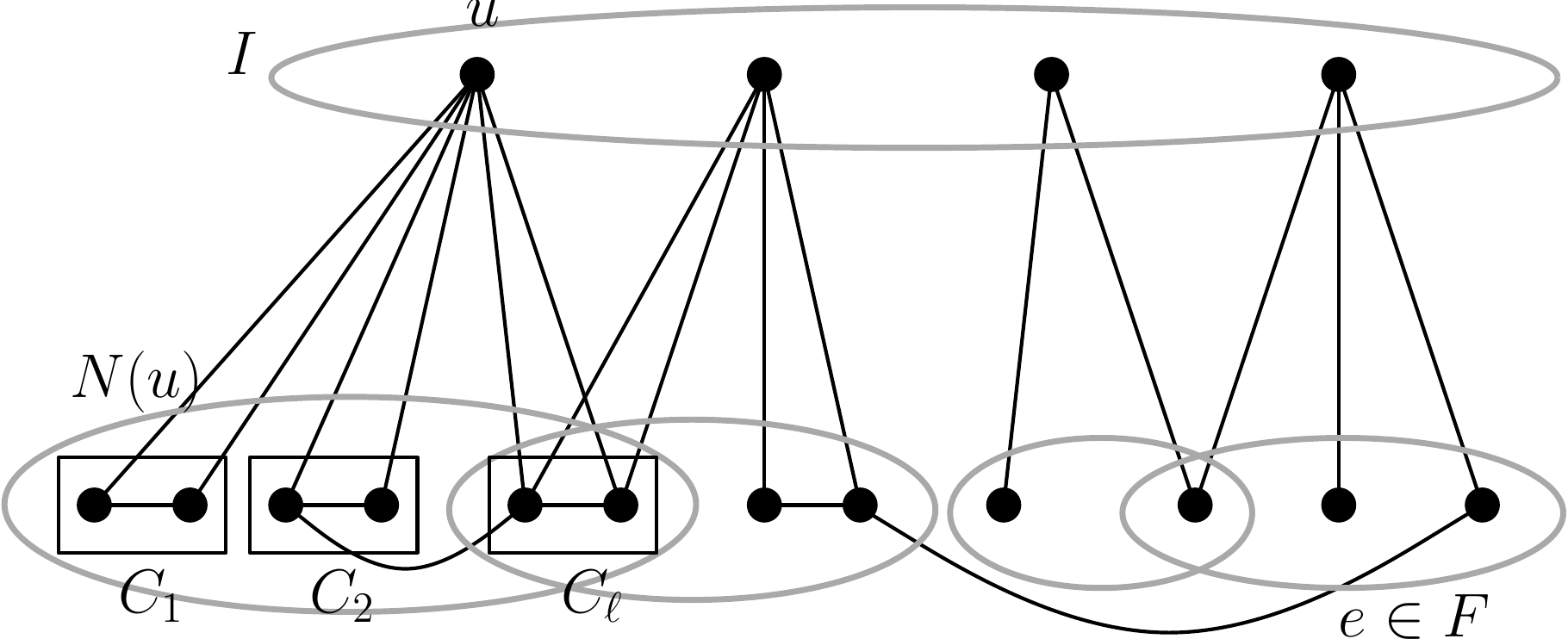}
\caption{A schema of the claw-free graph $G$.}\label{fig:clawfree}
\end{figure}
On the other hand, by virtue of \pref{thm:erdos}, for every $i, j$, $1\leq i<j\leq \ell$, the edges of $G$ which lie between the color classes $C_i, C_j$ can be partitioned into bicliques such that each vertex in $C_i\cup C_j$ is contained in at most $2\sqrt{2} c_0 \sqrt{\Delta\log \Delta}/\log \Delta$ of the bicliques. Since each $C_i$ induces a clique in $G$, the vertex set of each of these bicliques induces a clique in $G$. Hence,  all these cliques for every $1\leq i<j\leq \ell$,  cover all the edges in $G[N[u]]=G[N(u)\cup \{u\}]$. Let us denote this clique covering by $\mathcal{C}_u$.  In the clique covering $\mathcal{C}_u$, each vertex $v\in N[u]$ is contained in at most
\[2\sqrt{2} c_0  \frac{\sqrt{\Delta\log \Delta}}{\log \Delta}\ 2\sqrt{2}(1+o(1))\sqrt{\frac{\Delta}{\log \Delta}}=8c_0(1+o(1))\frac{\Delta}{ \log \Delta}\]
of the cliques. For each $u\in I$, let us cover the edges in $G[N[u]]$ by the clique covering $\mathcal{C}_u$ and define $\mathcal{C}:=\cup_{u\in I} \mathcal{C}_u$. Since $G$ is claw-free and $I$ is a maximal independent set, each vertex $v\in V$ has $1$ or $2$ neighbours in $I$. Therefore, every vertex $v\in V$ is contained in at most $16c_0(1+o(1))\Delta/\log \Delta$ of the cliques comprising $\mathcal{C}$. 

The cliques in $\mathcal{C}$ cover all the edges in $G[N[u]]$, for all $u\in I$, but it does not necessarily cover all the edges of $G$. Now, let $F\subset E$ be the subset of all edges which does not covered by the cliques of $\mathcal{C}$ and let $H$ be the subgraph of $G$ induced by $F$. For the remaining of the proof, we look for a suitable clique covering of $H$ and count up the contribution of each vertex in this covering.

In order to provide a desired clique covering for $H$, we have to describe the structure of the subgraph $H$. For this purpose, first, we establish a sequence of facts regarding  $H$.
%

Since all the edges covered by the cliques in $\mathcal{C}$ are exactly the ones in $G[N[u]]$, for all $u\in I$, we have the following fact.
\begin{description}
\item[Fact 1.]  For every edge $e=xy\in E$, we have  $e=xy\in F$ if and only if the vertices $x,y$ have no common neighbour in  $I$.
\end{description}
Assume $x\in V(H)$ and $y_1, y_2$  are two neighbours of $x$ in $H$, i.e. $y_1,y_2\in N_H(x)$. The vertex $x$ has also a neighbour $u$ in $I$. By Fact~1, $y_1,y_2$ are not adjacent with $u$. Hence, due to claw-freeness of $G$, $y_1,y_2$ must be adjacent in $G$. Consequently, the following fact holds. 
\begin{description}
\item[Fact~2.] For every vertex $x\in H$, its neighbours in $H$, $N_H(x)$ induces a clique in $G$.
\end{description}
 With the same argument (using Fact~1 along with the claw-freeness of $G$), one may prove the following fact.
\begin{description}
\item[Fact~3.] Every non-isolated vertex of $H$ has exactly one neighbour in the set $I$.
\end{description}

Assuming $I=\{u_1, \ldots, u_\alpha\}$, with the aid of Facts~1 and 3, the non-isolated vertices of $H$ can be partitioned into $\alpha$ disjoint sets $N_1,\ldots, N_\alpha$, where $N_i:=\{x\in V(H)\ :\  x \text{ is non-isolated and is adjacent to } u_i\}$, $1 \leq i \leq \alpha$.

Now suppose that $x \in V(H)$ and $y,z\in N_H(x)$, where $y \in N_i$ and $z \in N_j$, for some $i \neq j$. By Fact~2, we know that $y,z$ are adjacent in $G$. But $y,z$ has no common neighbour in $I$. Thus, due to Fact 1, $yz\in F$. Hence, the following assertion holds.
\begin{description}
\item[Fact~4.] If $x \in V(H)$ and $y,z\in N_H(x)$, where $y \in N_i$ and $z \in N_j$, for some $i \neq j$, then $yz$ is an edge in $F$.
\end{description}
Now, we are ready to prove the following claim concerning the structure of the graph $H$.
\begin{description}
\item[Claim.] Every connected component of $H$ is either
\begin{itemize}
\item an isolated vertex, or
\item a bipartite graph with bipartition $(N_i,N_j)$, for some $1\leq i<j\leq \alpha$, or
\item a graph on at most $2\Delta$ vertices whose diameter is at most $2$.
\end{itemize}
\end{description}
\begin{proof}[\textbf{Proof of Claim.}]
\renewcommand{\qedsymbol}{$\blacksquare$}
Fix a vertex $x\in N_i$ and for $d\in \mathbb{Z}^+\cup\{0\}$, let  $D_d:=\{y \in V(H) \ :\  d_H(x,y)=d\}$. First, we prove that if $D_d\subseteq N_j$, for some $d,j$, then $D_{d+2}\subseteq N_j$. To see  this, assuming $D_d\subseteq N_j$, let $y\in D_{d+2}$. Then, $y$ has a neighbour $y'$ in $D_{d+1}$ and $y'$ has a neighbour $y''$ in $D_d$, where $y''\in N_j$. If $y\notin N_j$, then, due to Fact~4, $y,y''$ should be adjacent in $H$, which is a contradiction. Thus, $y\in N_j$.

Hence, since $D_0=\{x\}\subseteq N_i$, we have $D_{2d}\subseteq N_i$, for every $d$. If, in addition, $D_1\subseteq N_j$, for some $j$, then $D_{2d+1}\subseteq N_j$, for all $d$. This shows that, in case $D_1\subseteq N_j$, the connected component of $H$ containing $x$, is a bipartite graph with bipartition $(N_i,N_j)$.

\begin{figure}
\centering
\includegraphics[scale=.65]{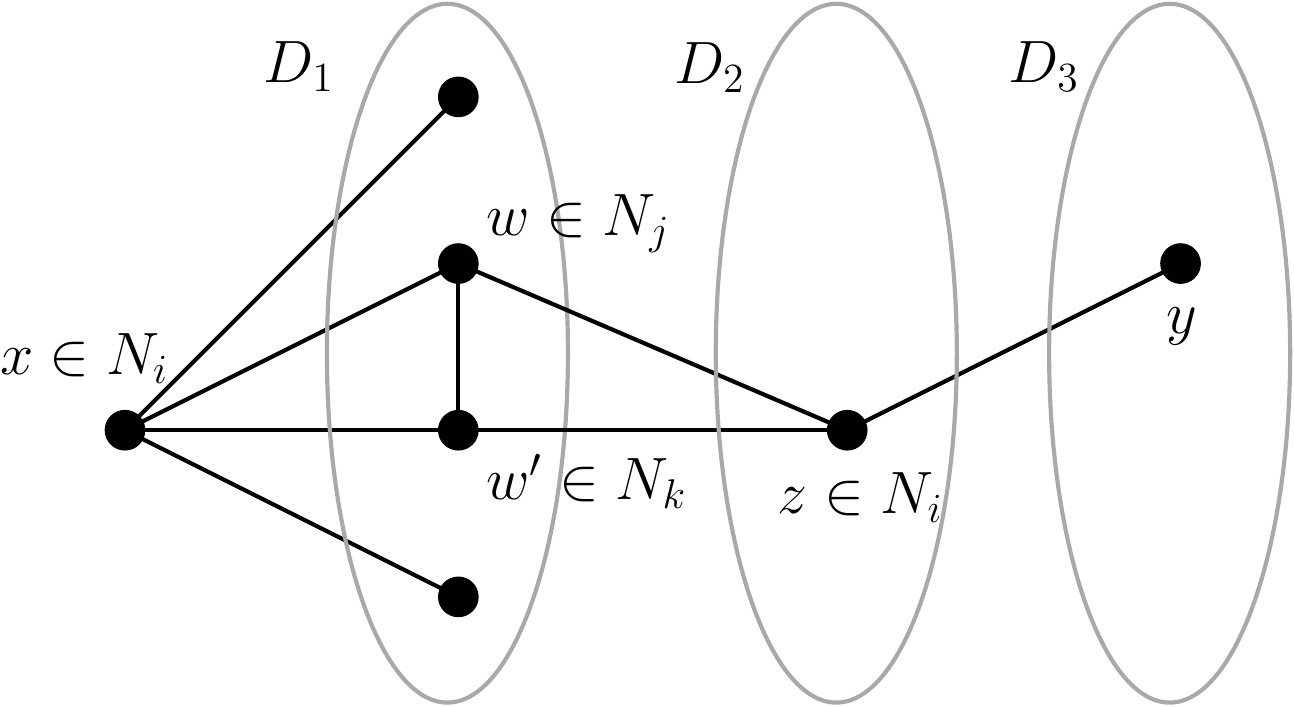}
\caption{A schema of a connected component of $H$, assuming, to the contrary, that $D_3\neq \emptyset$.}\label{fig:H}
\end{figure}
Now assume that $D_1\not \subseteq N_j$, for all $j$. With this assumption, we prove that the set $D_3$ is empty. Assume, to the contrary, that $y\in D_3$ and let $z$ be a neighbour of $y$ in $D_2$ and also let $w$ be a neighbour of $z$ in $D_1$ (See Figure~\ref{fig:H}). We have $w\notin N_i$ (because of Fact~1).  Assume $w\in N_j$, for some $j\neq i$. Since $D_1\not \subseteq N_j$, there is a vertex $w'\in D_1\cap N_k$, for some $k\neq j$. Due to Fact~4, $w$ and $w'$ are adjacent in $H$. Moreover, $z\in N_i$, again by Fact~4, $z$ and $w'$ are adjacent in $H$. Now, $y\in D_3$ are not adjacent to $w,w'\in D_1$ in $H$. Hence, by Fact~4, $y\in N_j\cap N_k$. This contradicts with the fact that $N_j$ and $N_k$ are disjoint. Consequently, the set $D_3$ is empty and the connected component of $H$ containing $x$ has diameter at most 2. On the other hand, $|D_1|= |N_H(x)|\leq \Delta$ and $|D_0\cup D_2|\leq |N_i|\leq \deg(u_i)\leq \Delta$. Hence, the connected component has at most $2\Delta$ vertices.
\end{proof}
Now, we get back into the proof of \pref{thm:claw}. Consider a nontrivial connected component of the graph $H$. By the above claim, either it is a graph on at most $2\Delta$ vertices whose diameter is at most 2, or it is a bipartite graph with bipartition $(N_i,N_j)$, for some $1\leq i<j\leq \alpha$. Since $|N_i|\leq \deg(u_i)\leq \Delta$, the latter case has also at most $2\Delta$ vertices. Hence, every connected component of $H$ is of size at most $2\Delta$. Therefore, by Theorem~\ref{thm:erdos}, one may construct   a biclique covering for    every connected component of $H$ where each of its vertices belongs to at most $2c_0\Delta/\log \Delta$ bicliques. Nevertheless, because of Fact~2, every biclique of $H$ induces a clique in $G$. As a consequent, one may provide a collection of cliques of $G$ which cover all the edges of $H$ and each vertex belongs to at most  $2c_0\Delta/\log \Delta$ of these cliques. This collection together with the clique covering $\mathcal{C}$ provides a clique covering for $G$, for which every vertex contributes in    at most $18c_0(1+o(1))\Delta/\log \Delta$ number of cliques,  thereby establishing Theorem~\ref{thm:claw}.
\end{proof}
%

%

%
%
%

\section{Linear interval graphs}\label{sec:lig}

In this section, we restrict ourself to a class of claw-free graphs, called linear interval graphs. A linear interval graph $G$ is constructed as follows. Consider a line and designate some points on it as the vertices of $G$. Also choose some intervals of the line (by an interval, we mean a proper subset of the line homeomorphic  to $[0,1]$). A pair of vertices $x,y$ being adjacent in $G$ if and only if they are both included in at least one of the intervals. A \textit{circular interval graph} is defined in the same way, taking a circle instead of a line. All such graphs are claw-free, and indeed they are quasi-line graphs. Linear and circular interval graphs play an important role in characterization of claw-free graphs \cite{seymour}.

In order to investigate local clique covering of linear interval graphs, we begin by a key  example. This example will take a significant role in analysing the general setting of linear interval graph. Let us define a linear interval graph on $2n$ vertices, $1,2,\ldots, 2n$, where the selected  intervals are  $[i,n+i]$, for all $1\leq i\leq n$. This graph is composed of a pair of cliques $X=\{1,\ldots, n\}$ and $Y=\{n+1,\ldots, 2n\}$ and each vertex $i\in X$ is adjacent to vertices $n+1, \ldots, n+i$ in $Y$. We denote this graph  by $K^\nabla_{n,n}$ and it is shown in Figure~\ref{fig:t5}, for $n=5$.
\begin{figure}
\centering
\includegraphics[scale=.4]{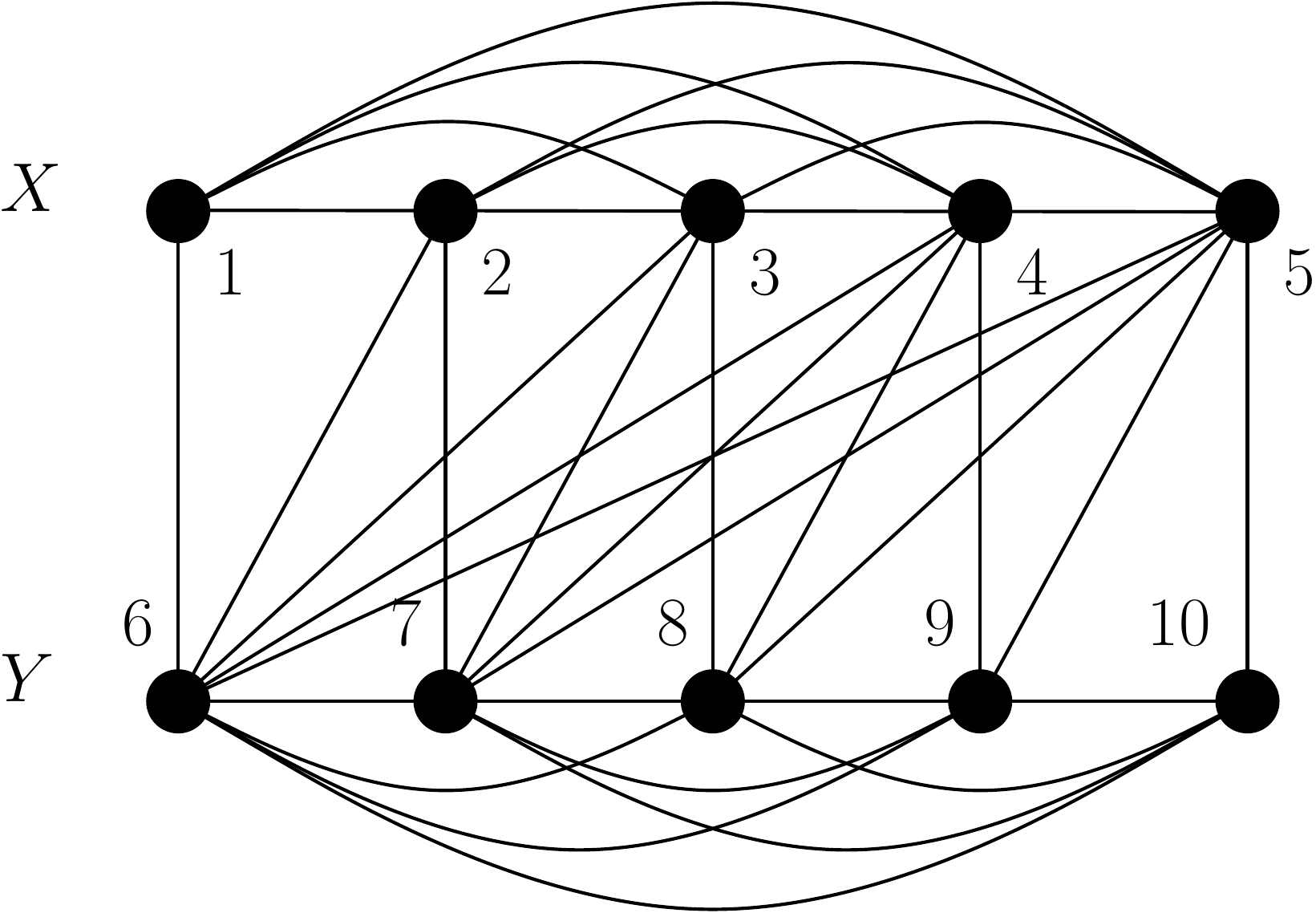}
\caption{The linear interval graph $K^\nabla_{5,5}$.}
\label{fig:t5}
\end{figure}

In the following theorem we prove that $\lcc{K^\nabla_{n,n}}$ is logarithmic in $n$.
\begin{thm}\label{thm:Gn}
For every positive integer $n$, $\lcc{K^\nabla_{n,n}}= \frac{1}{2}\log n+\frac{1}{4}\log \log n+O(1)$.
\end{thm}
First, note that if we neglect the edges inside $X$ and $Y$, we obtain a bipartite graph, $G_n$ with bipartition $(X,Y)$. Since $X$ and $Y$ are cliques, clique coverings of $K^\nabla_{n,n}$ are corresponding to biclique coverings of $G_n$ and $\lcc{K^\nabla_{n,n}}$ differs from $\lbc(G_n)$ by at most $1$ (for the definition of lbc, see Section~\ref{sec:intro}). As a result, it is enough to compute $\lbc(G_n)$, or equivalently the local boolean rank of its bipartite adjacency matrix, denoted by  $T_n$ (see Section~\ref{sec:intro}). The matrix $T_n$ is the lower triangular matrix whose all lower diagonal entries are equal to one, i.e.
\begin{equation}
\label{eq:tn}
T_n(i,j):=\begin{cases}
0& i< j, \\
1 & i\geq j.
\end{cases}
\end{equation}

By the definition, the local boolean rank of $T_n$ is the  smallest integer  $k$ for which we can cover the one entries of $T_n$ by rectangles (i.e. all-ones submatrices), such that each row and column intersects at most $k$ rectangles. 

Hence, in order to establish Theorem~\ref{thm:Gn}, it is enough to prove that $\lr_B(T_n)=\frac{1}{2}\log n+\frac{1}{4}\log \log n+O(1)$. Our proof is based on an algebraic discussion. Nevertheless, it is more illustrative, if one may find a constructive proof. To give an intuition to the problem, prior to getting into the proof, let us propose a simple recursive covering of the matrix $T_n$ by rectangles in which every row and column intersects at most $\log_3 n+O(1)$ rectangles. This ensures that $\lcc{K^\nabla_{n,n}}\leq \log_3 n+O(1)$, which is slightly worse than the best bound. The covering process is depicted in Figure~\ref{fig:Tn}. All rectangles are shown by boxes and shaded triangles are not covered yet (Figure~\ref{fig:Tn} left). Then we treat each shaded triangle as a new smaller matrix and recursively cover it exactly in the same way. Proceeding this process, until all shaded triangles become $1\times 1$, yields to a complete covering of $T_n$ (Figure~\ref{fig:Tn} right).

\begin{figure}
\centering
\includegraphics[scale=.35]{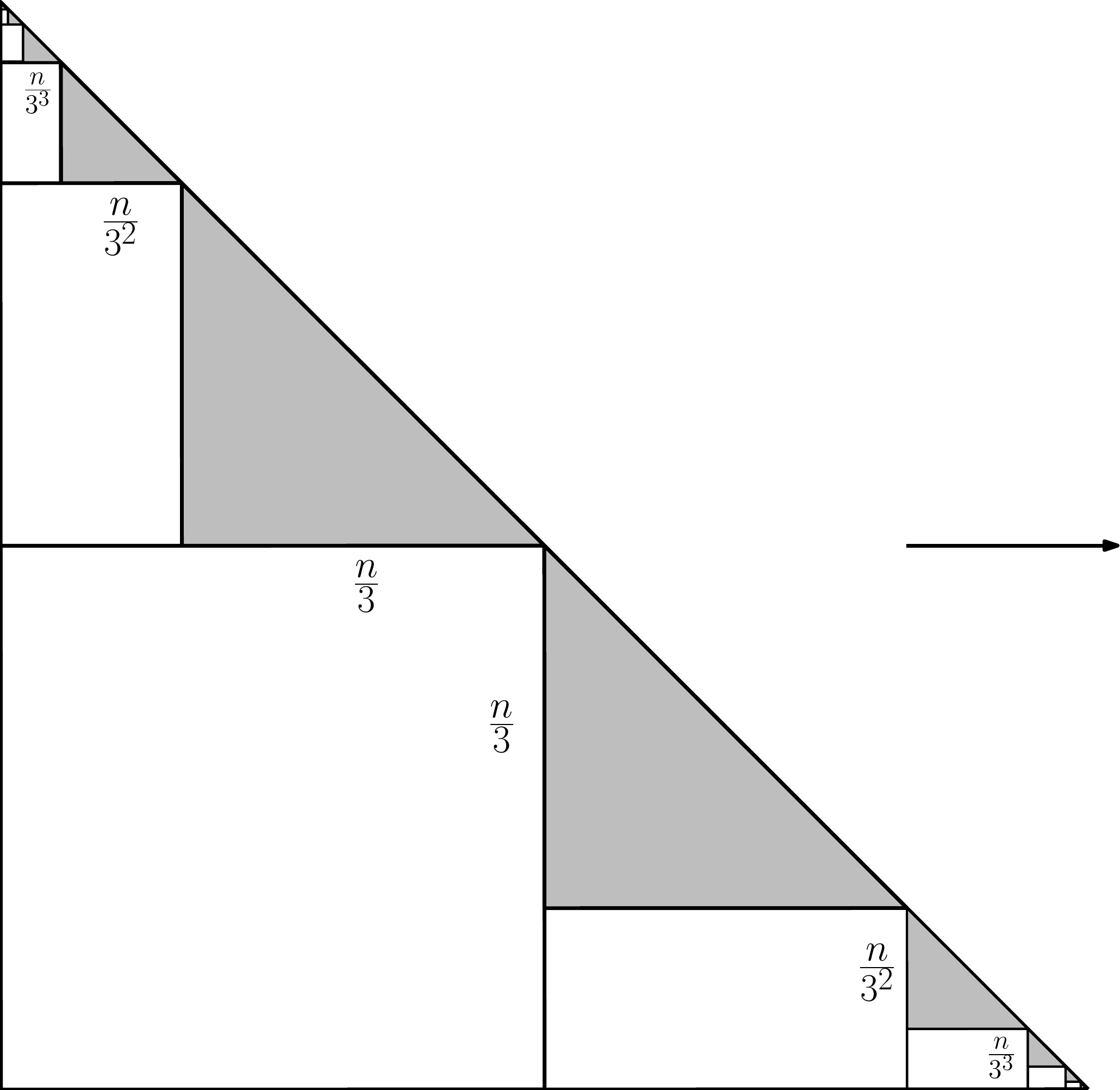}\qquad
\includegraphics[scale=.35]{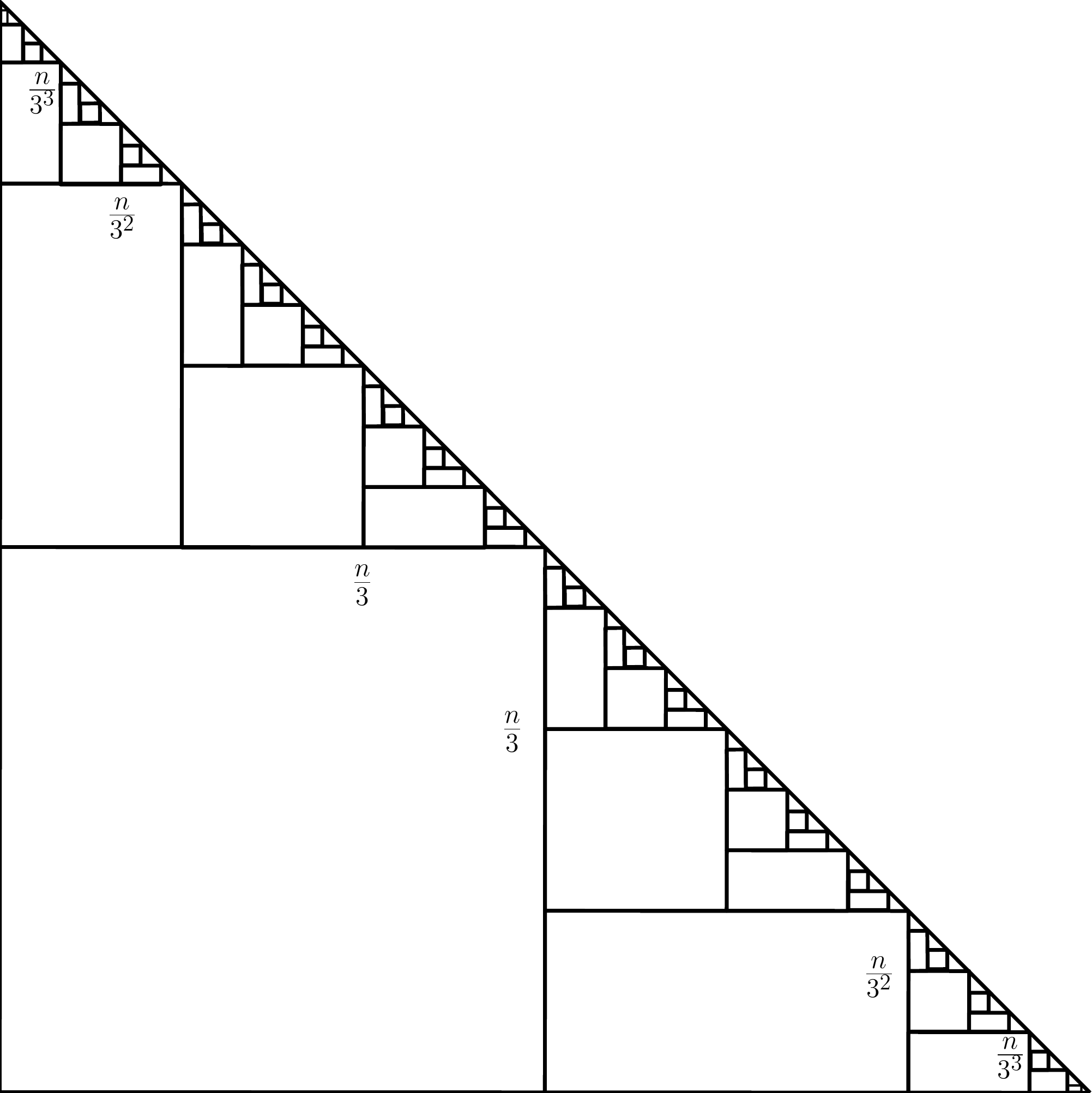}
\caption{Covering the matrix $T_n$ using  $\log_3 n+O(1)$ rectangles.}
\label{fig:Tn}
\end{figure}
Now we get into the proof of Theorem~\ref{thm:Gn}.
\begin{proof}[\textbf{Proof of Theorem~\ref{thm:Gn}.}]
Following the argument above, we prove that the lower triangle matrix $T_n$ can be covered by rectangles, in a way that each row and column contributes in at most $\frac{1}{2} \log n+\frac{1}{4} \log n+O(1)$ rectangles.

For every covering $\mathcal{C}$ of a matrix by rectangles, we  define the \textit{row valency} of $\mathcal{C}$  to be  the maximum number of appearance of each row in the rectangles. Analogously, \textit{column valency} of $\mathcal{C}$ is defined to be the maximum number of appearance of each column in the rectangles.

Now, for positive integers $r,s$, we nominate the function $f(r,s)$  as the maximum number $n$ for which one can provide the matrix $T_n$ with a covering by rectangles where its row and column valency are at most $r$ and $s$, respectively, i.e. each row (resp. column) intersects at most r (resp. s) rectangles. It is easy to see that
\begin{equation}\label{eq:lccfrr}
\lr_B(T_n)=\min\{ r\ :\  n\leq f(r,r)\}.
\end{equation}

Clearly,  $f(r,1)=r$ and $f(1,s)=s$. In addition, it satisfies the following recurrent relation.
\begin{equation}\label{eq:rec}
f(r,s)= f(r-1,s)+f(r,s-1), \quad r,s>1.
\end{equation}
To see why the recurrent relation holds, consider a covering of $T_n$ for $n=f(r,s)$. Without loss of generality, we can assume that the rectangle containing the entry $(1,1)$ also contains a diagonal entry (otherwise, we can extend the rectangle, without increasing the row and column valency). Thus, the matrix $T_n$ is divided into two submatrices $T_{n_1}$ and $T_{n_2}$ (Figure~\ref{fig:rec}), where the row (resp. column) valency of $T_{n_1}$ is at most $r-1$ (resp. $s$) and the row (resp. column) valency of $T_{n_2}$ is at most $r$ (resp. $s-1$). This guarantees that $n_1\leq f(r-1,s)$ and $n_2\leq f(r,s-1)$, concluding $n=f(r,s)\leq f(r-1,s)+f(r,s-1)$. The reverse inequality trivially holds.

\begin{figure}
\centering
\includegraphics[scale=.3]{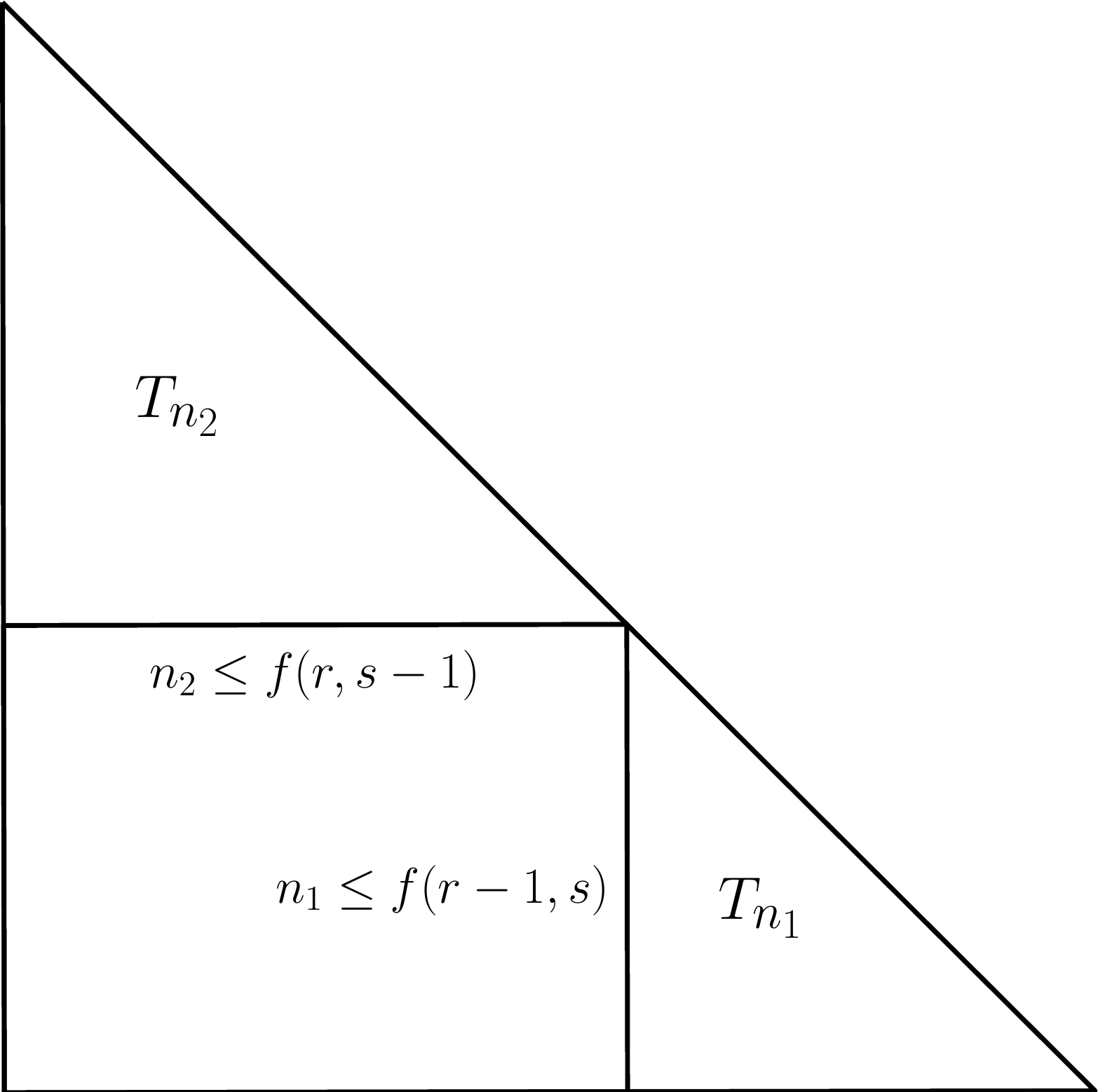}
\caption{The schematic proof of the recurrent relation \eqref{eq:rec}.}
\label{fig:rec}
\end{figure}
Now, in order to have an explicit formula for $f(r,s)$, for fixed integer $s$, let  $g_s(x)=\sum_{r=1}^\infty f(r,s) x^r$ be the generating function of $f(r,s)$.
Then, using \eqref{eq:rec}, we have
\begin{align*}
g_1(x)&= \sum_{r=1}^\infty r x^r=\frac{x}{(1-x)^2},\\
(1-x)g_s(x)&= g_{s-1}(x)+x, \quad s>1.
\end{align*}
Solving this recurrent relation implies that
\begin{align*}
g_s(x)&= \frac{x}{(1-x)^{s+1}}+\frac{1}{(1-x)^{s-1}}-1\\
&= \sum_{r=0}^\infty \binom{r+s}{r} x^{r+1} +\sum_{r=0}^\infty \binom{r+s-2}{r} x^r-1\\
&=\sum_{r=1}^\infty \left[\binom{r+s-1}{r-1} +\binom{r+s-2}{r}\right] x^r.
\end{align*}
We conclude that $f(r,s)$ satisfies the following explicit formula
\begin{equation}
\label{eq:frs}
f(r,s)=\binom{r+s-1}{r-1} +\binom{r+s-2}{r}, \quad \forall r,s\in \mathbb{Z}^+.
\end{equation}
Hence
\begin{align*}
f(r,r)&=\binom{2r-1}{r-1} +\binom{2r-2}{r}\\ 
&= \frac{1}{2}\binom{2r}{r}+\frac{r-1}{r} \binom{2r-2}{r-1}\sim \frac{1}{2} \frac{4^r}{\sqrt{\pi r}}+ \frac{4^{r-1}}{\sqrt{\pi(r-1)}}\sim \frac{3\times 4^{r-1}}{\sqrt{\pi (r-1)}}.
\end{align*}
This, together with \eqref{eq:lccfrr}, yields
\[\lr_B(T_n)=\log_4 n+\frac{1}{2} \log_4\log_4 n+O(1)=\frac{1}{2}\log n+\frac{1}{4} \log \log n+O(1),\]
as desired.
\end{proof}

Now we are ready to deal with the general case of linear interval graphs. First of all, we set a couple of assumptions. Note that, in the definition of the linear interval graphs, if there is some interval containing another interval, we can exclude smaller one, without making any change in the structure of the graph. For this reason, henceforth, we assume that no interval is contained in another one. Moreover, if there is no vertex between each two consecutive opening (closing) brackets of the intervals, then we can move one of the brackets such that one interval comes to be contained in another. Also, if there is no vertex between an opening bracket of one interval and closing bracket of another one, then we can move one of the brackets to make these intervals disjoint. Since these modifications do not change the adjacencies in the graph, henceforth, we assume that there is at least one vertex between each two consecutive opening (closing) brackets and at least one vertex between opening bracket of an interval and closing bracket of another interval.

\begin{thm} \label{thm:lig}
For every linear interval graph $G$ with maximum degree $\Delta$, we have $\lcc{G}\leq \log \Delta+\frac{1}{2}\log \log \Delta+O(1)$. Furthermore, if $G$ is twin-free, then 
\[\frac{1}{2} \log\Delta+\frac{1}{4} \log\log\Delta +O(1)\leq \lcc{G}\leq \log\Delta +\frac{1}{2}\log\log\Delta+O(1).\]
\end{thm}
\begin{proof}
Given a linear interval graph $G=(V,E)$,  we begin by  the leftmost interval and call it $I_1$. Then we  consider the first interval opening after closing $I_1$ and call it $I_2$. Continue the labelling  until the interval $I_k$ after which no interval opens. The intervals $I_1,\ldots,I_k$ are disjoint and since no interval contains another (see the note just before Theorem~\ref{thm:lig}), every other interval $I$ opens inside one $I_i$ and intersects at most $I_i, I_{i+1}$. 

For every  $i$, $1\leq i\leq k$, let $\mathcal{I}_i$ be the set of all intervals which are opened inside $I_i$, including $I_i$ itself. Also, define
\[V_i:=\{x\in V\ :\  x\in I, \mbox{ for some } I\in\mathcal{I}_i\},\]
and let $G_i$  be the induced subgraph of $G$ on $V_i$. Every non-isolated vertex $x\in V$ is in at least one $V_i$ and at most two $V_i$'s, because 
\begin{align*}
&I_i\cap V_j =\emptyset & \forall\ j\geq i+1, \mbox{ or } \forall\  j\leq i-2, \\
&(V_i\setminus I_i) \cap V_j=\emptyset & \forall\ j\geq i+2, \mbox{ or } \forall\ j\leq i-1.
\end{align*}
Therefore, every vertex of $G$ lives in at most two of the graphs $G_i$'s. This, along with the fact that $E(G)=\cup_i E(G_i)$, yields
\[\lcc{G}\leq 2\ \max_i \lcc{G_i}.\]
Now fix $i$ and let us compute $\lcc{G_i}$. First, note that by Remark~\ref{rem:twinfree}, $G_i$ has an induced twin-free subgraph $G'_i$, where $\lcc{G_i}=\lcc{G'_i}$. The fact that $G'_i$ is twin-free, along with the assumptions adopted just before Theorem~\ref{thm:lig}, conclude that there exists exactly one vertex  in $G'_i$ between each two consecutive opening or closing brackets of intervals (if there were two vertices, they would be twins). Accordingly, if $t$ is the number of intervals in $\mathcal{I}_i$, then $G'_i$  has $2t-1$ vertices and is isomorphic to the graph $K^\nabla_{t,t}\setminus (t+1)$, which is the graph obtained from $K^\nabla_{t,t}$ by deleting vertex $t+1$. To see this, call the vertices of $G'_i$ by $y_1,\ldots, y_{2t-1}$, in order from left to right and define an isomorphism mapping $y_i$ to $i$, for $1\leq i\leq t$, and to $i+1$, for $t+1\leq i\leq 2t-1$, respectively.

On the other hand, the vertices $t$ and $t+1$ are twins in $K^\nabla_{t,t}$. Hence, by Theorem~\ref{thm:Gn}, we have
\[\lcc{G_i}=\lcc{K^\nabla_{t,t}\setminus (t+1)}=\lcc{K^\nabla_{t,t}}\leq \frac{1}{2}\log \Delta+ \frac{1}{4}\log \log \Delta+O(1),\]
which yields the desired upper bound. 

For the lower bound, let $G$ be twin-free and define $o_i$ and $c_i$ to be the number of opening and closing brackets inside the interval $I_i$, respectively. Also let $N:=o_{i_0}=\max_i o_i$.  Since $G$ is twin-free, at most one vertex lies between each two consecutive opening or closing brackets. Therefore, for every vertex $v\in V_i$, we have
\[\deg_G(v)\leq o_{i-1} +c_{i-1} +o_{i} +c_{i} +o_{i+1} +c_{i+1}\leq 6 N. \]
The last inequality holds due to the fact that $c_i\leq o_{i-1}$, for all $i$. Let $G_{i_0}$ and $G'_{i_0}$ be the same as defined earlier. Thus, $G'_{i_0}$ is an induced subgraph of $G$ which is isomorphic to the graph $K^\nabla_{N,N}\setminus (N+1)$. Furthermore, $N \geq \Delta/6$ and therefore, by Theorem~\ref{thm:Gn},
\[\lcc{G}\geq \lcc{G'_{i_0}} \geq \frac{1}{2} \log N + \frac{1}{4} \log\log N + O(1) \geq \frac{1}{2} \log \Delta + \frac{1}{4} \log\log \Delta + O(1).\]
\end{proof}
\section{A Bollob\'as-type inequality}\label{sec:bol}
We close the paper with a result in extremal set theory. In fact, we apply a result of Section~\ref{sec:lig} to prove a new Bollob\'as-type inequality concerning the intersecting pairs of set systems.

An $(r,s)-$system is defined as a family $\mathcal{F}=\{(A_i,B_i)\ : \ i\in I\}$ of pairs of sets, such that for every $i\in I$, $|A_i|\leq r$, $|B_i|\leq s$ and $A_i\cap B_j\neq \emptyset$ if and only if $i\neq j$. A classical result of Bollob\'as \cite{bollobas} known as Bollob\'as's inequality states that for every $(r,s)-$system $\mathcal{F}$, $|\mathcal{F}|\leq \binom{r+s}{r}$.

A number of generalizations of this result have been presented in the context of extremal set theory. A remarkable extension is due to Frankl \cite{frankl}, that states a skew version of Bollob\'as's result. Let  $\mathcal{F}=\{(A_i,B_i)\ : \ 1\leq i\leq k\}$ be a family of pairs of sets, where for every $i$, $1\leq i\leq k$, $|A_i|\leq r$, $|B_i|\leq s$, also $A_i\cap B_i=\emptyset$  and for every $i,j$, $1\leq i<j\leq k $, $A_i\cap B_j\neq \emptyset$. In this case, it is proved by Frankl that $|\mathcal{F}|\leq \binom{r+s}{r}$. 

Another variant of the problem is proposed by Tuza \cite{tuza} as follows.  An $(r,s)-$weakly intersecting pairs of sets is defined as a family $\mathcal{F}=\{(A_i,B_i)\ : \ i\in I\}$, such that for every $i\in I$, $|A_i|\leq r$, $|B_i|\leq s$, $A_i\cap B_i=\emptyset$  and  for every $i\neq j$, the sets $A_i\cap B_j$ and $A_j\cap B_i$ are not both empty. As far as we know, the best known upper bound is $|\mathcal{F}|< (r+s)^{r+s} / (r^r s^s)$, due to Tuza \cite{tuza}. For a review of the other generalizations see \cite{talbot} and references therein.

Now, let us define a new variant of the problem as follows. Let $\mathcal{F}=\{(A_i,B_i)\ : \ 1\leq i\leq k\}$ be a family of pairs of sets, satisfying the following conditions.
\begin{enumerate}
\item[(i)] For every $i$, $1\leq i\leq k$, $|A_i|\leq r$, $|B_i|\leq s$.
\item[(ii)] For every $i,j$, $1\leq i,j\leq k$, $A_i\cap B_j\neq \emptyset$ if and only if $i\geq j$. 
\end{enumerate}
We are seeking for the maximum size of such a family $\mathcal{F}$. To this end, we recall the definition of the function $f(r,s)$ in the proof of Theorem~\ref{thm:Gn}. For every positive integers $r,s$, $f(r,s)$ is defined to be the maximum number $n$, such that there exists a covering of one entries of the matrix, $T_n$, defined in \eqref{eq:tn},  by rectangles (i.e. all-ones submatrices), where each row (resp. column) intersects at most $r$ (resp. $s$) of the rectangles.

We claim that $f(r,s)$ is indeed equal to the maximum possible size of a family $\mathcal{F}$ satisfying the above conditions (i) and (ii). For this, consider a  rectangle covering $\mathcal{C}$ of $T_n$, where each row (resp. column) intersects at most $r$ (resp. $s$) of rectangles. Now, for every $1\leq i,j\leq n$, let $A_i$ and $B_j$ be the sets of all rectangles in $\mathcal{C}$ intersecting row $i$ and column $j$, respectively. It is evident that $|A_i|\leq r$ and $|B_j|\leq s$. Also, one may see that $A_i\cap B_j\neq \emptyset$ if and only if  there exists a rectangle in $\mathcal{C}$ intersecting both row $i$ and column $j$ and then if and only if $T_n(i,j)=1$, i.e. $i\geq j$.

Conversely, assume that a family $\mathcal{F}=\{(A_i,B_i)\ :\ 1\leq i\leq k\}$ satisfies the above conditions (i) and (ii). Then, for every $x\in \cup_i (A_i\cup B_i)$, define  $C_x:=\{(i,j)\ :\ x\in A_i\cap B_j \}$. Thus, one may check that all these sets $C_x$ form a rectangle covering for the matrix $T_n$ and for every $1\leq i,j\leq k$, row $i$ (resp. column $j$) intersects at most $|A_i|$ (resp. $|B_j|$) of the rectangles thereby obtained.

The above arguments show that the maximum possible size of a family $\mathcal{F}$ fulfilling the conditions (i) and (ii) is indeed equal to $f(r,s)$.  The explicit formula for $f(r,s)$ has been obtained in \eqref{eq:frs}. This enables us to state the following result which is an extremal inequality analogous to Bollob\'as's inequality regarding a pair of set systems.
\begin{cor}
If $\mathcal{F}=\{(A_i,B_i)\ : \ 1\leq i\leq k\}$ be a family of pairs of sets fulfilling the conditions (i) and (ii), then $|\mathcal{F}|\leq \binom{r+s-1}{r-1}+\binom{r+s-2}{r}$. Also, the upper bound is tight.
\end{cor}

%

\end{document}